\documentclass[letterpaper, 10 pt, conference]{ieeeconf}  

\IEEEoverridecommandlockouts                              

\usepackage{graphics} 
\usepackage{epsfig} 
\usepackage{amsmath} 
\usepackage{amssymb}  

\usepackage{amsthm}
\usepackage[ruled,vlined,linesnumbered]{algorithm2e}

\newtheorem{problem}{Problem}
\newtheorem{definition}{Definition}

\newtheorem{corollary}{Corollary}
\newtheorem{lemma}{Lemma}
\newtheorem{theorem}{Theorem}

\title{\LARGE \bf
Planning for Package Deliveries in Risky Environments Over Multiple Epochs
}

\author{Blake Wilson, Jeffrey Hudack and Shreyas Sundaram
\thanks{B. Wilson, and S. Sundaram are with the Elmore Family School of Electrical and Computer Engineering,
        Purdue University, West Lafayette, IN 47907 USA
        {\tt\small wilso692@purdue.edu, ye159@purdue.edu,sundara2@purdue.edu}}%
\thanks{J. Hudack is with the Air Force Research Lab, Rome, NY 13441, USA
        {\tt\small jeffrey.hudack@us.af.mil}}%
}

\begin{document}

\maketitle
\thispagestyle{empty}
\pagestyle{empty}

\begin{abstract}

We study a risk-aware robot planning problem where a dispatcher must construct a package delivery plan that maximizes the expected reward for a robot delivering packages across multiple epochs. Each package has an associated reward for delivery and a risk of failure. If the robot fails while delivering a package, no future packages can be delivered and the cost of replacing the robot is incurred. The package delivery plan takes place over the course of either a finite or an infinite number of epochs, denoted as the finite horizon problem and infinite horizon problem, respectively. The dispatcher has to weigh the risk and reward of delivering packages during any given epoch against the potential loss of any future epoch's reward. By using the ratio between a package's reward and its risk of failure, we prove an optimal, greedy solution to both the infinite and finite horizon problems. The finite horizon problem can be solved optimally in $O(K n\log n)$ time where $K$ is the number of epochs and $n$ is the number of packages. We show an isomorphism between the infinite horizon problem and Markov Decision Processes to prove an optimal $O(n)$ time algorithm for the infinite horizon problem.

\end{abstract}



\section{INTRODUCTION}
Robotic platforms (such as drones)  hold incredible promise for tasks such as reconnaissance, search and rescue, and delivery of essential items in dangerous environments (e.g., during natural disasters, or in battlefields). In recent decades, significant focus has been put toward path planning for robots by solving graph theoretic problems \cite{LAPORTE1992345}. A common goal for these problems is to collect rewards (e.g., information or objects) by visiting physical locations. These problems are known as Orienteering Problems (OP) and many variants are applicable to realistic engineering problems, such as environmental inspection \cite{Katrasnik_Power_2009}, multi-agent search \cite{2017_Jorgensen_Pavone_Matroid_Team_Orienteering}, and sensor data collection \cite{Hudack2016}. Typically, variants of OP are NP-Hard, and  must be solved via heuristics like Monte Carlo methods, and linear programming \cite{2016_03_15_Gunawan_Orienteering, LAPORTE1992345, LAPORTE1990193, Tsiligirides1984}. Up until recently, one aspect often missing from these formulations is that robots come at a cost, and they face the risk of failure when deployed in dangerous environments. Therefore, to maximize the effectiveness of a limited supply of robots in such settings, one needs to weigh the reward of achieving a given task against the risk of robot failure and cost of replacement. 

To account for this, one can construct a variant of OP where each edge in the graph has an associated risk of failure. Then, rather than finding a tour which maximizes the collected reward, the goal becomes to find the tour that maximizes the expected reward, i.e., risk-aware OP \cite{2017_Jorgensen_Pavone_Matroid_Team_Orienteering,Schwager2017}. These kinds of problems are crucial for logistic operations in regions with adversaries, such as maritime transportation \cite{VANEK2013157}. Recently, it was shown that the solution to some team risk-aware OP problems have a matroid structure which allow for matroid optimization techniques \cite{2017_Jorgensen_Pavone_Matroid_Team_Orienteering}. However, if the problem introduces some cyclic behavior like a return depot, then the cycles can break the matroid structure and other optimization techniques are required \cite{Prasad_A_2019_Risk_Aware}. The return depot formulation encompasses many real-world applications like search and rescue, sensor data collection, and package delivery. 


In this paper, we focus on a particular OP where a single agent delivers packages (e.g., containing aid or other essential items) from a depot to a set of locations. Delivering a package to a given location earns a reward (specific to that package and location), but also has a (location-specific) risk of failure for the agent that is delivering the package.  Furthermore, packages are desired to be delivered to the various locations repeatedly (i.e., over a set of {\it epochs} representing, for example, hours, days, or weeks). Each epoch has the same set of packages as any other epoch. Multi-epoch missions have been explored in the context of intermittent deployment \cite{2019JunCoupledTasks,2020JunIntermittent}. However, they have not been explored in the context of risk-aware task allocation before. For each epoch, a dispatcher determines the order of packages for the agent in order to maximize the expected reward from delivering packages while accounting for risk of failure along the way. The key challenge in this setting is to identify a rigorous strategy for dispatching the agent that accounts for all of these various features (drone costs, task-specific rewards, task-specific risks of failure, and the need to plan over multiple epochs). Additionally, we will consider both a finite horizon and an infinite horizon setting where the number of epochs is finite and infinite, respectively. For the finite horizon case, we prove an optimal $O(n \log n)$ time algorithm where $n$ is the number of packages. For the infinite horizon case, we map the problem to an isomorphic Markov Decision Process (MDP) and prove the optimal solution can be found in $O(n)$ time.

Our work differs from the previously mentioned studies in two main ways. Firstly, insofar as we are aware, maximizing the expected reward across multiple epochs has not been considered in any previous risk-aware OP variant. Secondly, while the single-agent variant we consider in this paper does form a matroid, our greedy solution is optimal without the need of matroid heuristics. 





\section{Problem Formulation}
We consider time as being measured by a set of $K$ epochs $\{1, 2, \ldots, K\}$, where each epoch can represent an hour, a day, a week, etc., as appropriate for the scenario.  There is a depot $t_0$ which contains a set of packages $\mathcal{T} = \{t_1, t_2, ..., t_n\}$ at the start of each epoch, where each package, $t_i \in \mathcal{T}$ is requested to be delivered to a certain location.  We assume that the package for each location is unique (i.e., a package for one location cannot be delivered to another location), and thus we use $t_i$ to denote both the package and the destination location. Furthermore, we assume that the set of packages $\mathcal{T}$ is replenished at the depot at the start of each epoch.  Each package $t_i \in \mathcal{T}$ has an associated reward  $r_i \in \mathbb{R}_{+}$ if it is successfully delivered to its target location in a given epoch. An agent (e.g., drone) is available at the depot at the first epoch to deliver the packages to their desired locations.  Since the agent may fail during an epoch in our problem formulation, for each $h \in \{1,2,...,K\}$ we define $\alpha_h = 1$, if the agent survives to epoch $h$ and $\alpha_h = 0$ otherwise. The agent can carry at most 1 package at a time; after it has delivered the package it is carrying, it must return to the depot to pick up another package. For each package $t_j$ that the agent attempts to deliver, there is a probability that the agent will fail either en route to the target location, or on the way back to the depot. We assume that the events representing successful agent traversal of each leg of the trip are independent and have equal probability, given by $\rho_j$ for some $\rho_j \in [0,1]$. Thus the probability that the agent successfully delivers package $t_j$ and returns to the depot is given by $\rho_j^2$.  We also assume that the events denoting successful agent traversal for different packages are independent.  If the agent fails while delivering a package, a cost of $\theta \in \mathbb{R}_{\ge 0}$ is incurred and the agent cannot deliver any more packages for the current or any future epoch.

 During each epoch that the agent is alive it executes an assigned delivery plan in some order specified by the dispatcher. Due to the fact that the agent can only carry one package at a time, a package delivery plan for epoch $h$ is represented by an ordered set of cycles $\mathcal{C}_{h} = \{ W_{h,1}, W_{h,2}, ..., W_{h,q_{h}} \}$ where each package delivery $W_{h,j} \in \mathcal{C}_{h}$ has the form $W_j = \{t_0, t_{h,j}, t_0\}$, with $t_{h,j} \in \mathcal{T}$ representing the $j$-th package being assigned for delivery by the agent in epoch $h$. 

  This ordered set may not contain every package delivery available for the epoch if the dispatcher considers some packages too risky to deliver, i.e. $q_h < |\mathcal{T}|$.  Let $\mathcal{C} = \{\mathcal{C}_1, \mathcal{C}_2, \ldots, \mathcal{C}_{K}\}$ denote a {\bf package delivery plan} for the entire $K$-epoch period of the mission.  Let $E(\mathcal{C}_h | \alpha_h)$ denote the expected reward provided by plan $\mathcal{C}_h$ for epoch $h$ conditioned on whether the agent survives to epoch $h$. The goal of the dispatcher is to maximize the expected reward across all epochs by constructing a package delivery plan. By the law of total expectation, the expected reward $E(\mathcal{C})$ across all epochs for a given package delivery plan $\mathcal{C}$ is:


\begin{equation}
    E(\mathcal{C}) = \sum_{h = 1}^{K} E(\mathcal{C}_h | \alpha_h = 1)P(\alpha_h = 1).
    \label{eq:total_expectation}
\end{equation}



Then, the dispatcher's goal is to solve the following problem:

\begin{problem}
 Risk-Aware Single-Agent Package Delivery (RSPD)

 \begin{equation}
    \text{Find   } \text{arg}\max_{\mathcal{C}_1,\ldots,\mathcal{C}_K} \sum_{h = 1}^{K} E(\mathcal{C}_{h} | \alpha_h = 1 )P(\alpha_h = 1).
    \label{eq:RSPD}
\end{equation}
\end{problem}

Consider the set of cycles $\mathcal{C}_{h} = \{W_{h,1}, W_{h,2}, \ldots, W_{q_h}\}$ representing the sequence of package deliveries assigned to the agent on epoch $h$. Assuming the condition $\alpha_h = 1$ (i.e., that the agent has survived to epoch $h$), the conditional probability $\bar{\rho}_{h,j}$ that the agent completes a given delivery $W_{h,j} \in \mathcal{C}_{h}$, and returns to base is dependent on the probability the agent survives every cycle prior to and including $W_{h,j}$, i.e., $\bar{\rho}_{h,j} = \rho_{h,1}^2\rho_{h,2}^2\cdots\rho_{h,j}^2$.  Let $\psi_{h,j}$ denote the probability that the agent successfully delivers the $j$-th package in epoch $h$, given by $\psi_{h,j} = \bar{\rho}_{h,j-1}\rho_{h,j}$. Then, the expected reward for $W_{h,j}$ is given by $r_{h,j} \psi_{h,j}$. Additionally, $P(\alpha_h = 1)$ is dependent on the probability of finishing the last cycle for each epoch before $h$.  So for $h > 1$, we can calculate $P(\alpha_h = 1)$ by: 
\begin{equation}
    P(\alpha_h = 1) = \prod_{l = 1}^{h-1}\bar{\rho}_{l,q_l} \label{eq:RSPD_q_h}
\end{equation}
where $\bar{\rho}_{l,q_l}$ is the probability of finishing the last cycle $W_{l,q_l} \in \mathcal{C}_l$. Because we are guaranteed the agent is alive at the start of the first epoch, we have $P(\alpha_1 = 1) = 1$. If the agent survives to epoch $h$, the conditional expected reward for epoch $h$ is calculated by:
\begin{equation}
     E(\mathcal{C}_h | \alpha_h = 1)  = \sum_{j = 1}^{q_h}r_{h,j} \psi_{h,j} - \theta (1 - \bar{\rho}_{h,q_h}). \label{eq:RSPD_1epoch_expectation}
\end{equation}
The first term captures the expected rewards of completing the ordered set of tasks in epoch $h$, while the second term captures the expected cost of losing the agent in epoch $h$.
Substituting (\ref{eq:RSPD_q_h}) into (\ref{eq:RSPD}) reveals a telescopic relationship: 
\begin{align}
    &\sum_{h = 1}^K E(\mathcal{C}_{h} | \alpha_h = 1 )P(\alpha_h = 1) = E(\mathcal{C}_{1} | \alpha_1 = 1 ) + \label{eq:sub_RSPD_general_form} \\ & \bar{\rho}_{1,q_1}\left(E(\mathcal{C}_{2} | \alpha_2 = 1 ) + \bar{\rho}_{2,q_2}(E(\mathcal{C}_{3} | \alpha_3 = 1 )  + \ldots)\right).  \nonumber
\end{align}

Equation (\ref{eq:sub_RSPD_general_form}) can be written recursively as follows.

\begin{definition}[Inductive Expected Reward for RSPD]
Given a package delivery plan $\mathcal{C} = \{\mathcal{C}_1, \mathcal{C}_2, \ldots, \mathcal{C}_K\}$ for RSPD, the reward functions $v_1, v_2, \ldots, v_K$  are defined recursively as follows:
\begin{equation}
\begin{split}
    v_K(\mathcal{C}_{K}) &\triangleq E(\mathcal{C}_{K}|\alpha_K = 1) \\
    v_h(\mathcal{C}_{h}) &\triangleq E(\mathcal{C}_{h}|\alpha_h = 1) + \bar{\rho}_{h,q_h} v_{h+1}(\mathcal{C}_{h+1}) 
\end{split}
\label{eq:inductive_expected_reward_RSPD}
\end{equation}
for $h \in \{1, 2, \ldots, K-1\}$.
\end{definition}

This recursive relationship means we can maximize the total expected reward by a Bellman equation that works backwards from epoch $K$:

\begin{align}
    V_h &= \max_{\mathcal{C}_{h}} [E(\mathcal{C}_{h}|\alpha_h = 1) + \bar{\rho}_{q_h} V_{h+1}], h \in \{1,2,...,K - 1 \}
    \label{eq:Bellman_eq_a1_mu1} \\
    V_K &= \max_{\mathcal{C}_{K}} E(\mathcal{C}_{K}|\alpha_K = 1).
    \label{eq:Bellman_eq_base_a1_mu1}
\end{align}

Now that we have the Bellman equation for solving RSPD, we will focus on the characteristics of optimal package delivery plans that solve Equations (\ref{eq:Bellman_eq_a1_mu1}) and (\ref{eq:Bellman_eq_base_a1_mu1}).

\section{Optimal Package Delivery Plans}
For any given package delivery plan, the dispatcher decides which packages to assign in each epoch and in what order. Naturally, when comparing two packages $t_i$ and $t_j$, the dispatcher should pick the package with the highest reward and the lowest probability of failure first, i.e., $r_i > r_j$ and $\rho_i > \rho_j$. However, there will be many cases where $r_i > r_j$ while $\rho_i < \rho_j$, i.e., the most valuable packages are also the most risky. Hence, the dispatcher needs to compare package deliveries by a function of each package's reward and its probability of failure. For each package delivery, we will define its \textbf{reward-to-risk ratio} as follows, and later show that it will serve as the primary means of comparison between packages.

\begin{definition}[Reward-to-Risk Ratio]
 Given a package delivery cycle $W_j = \{t_0,t_j,t_0\}$ for a package $t_j \in T$, let the reward of the package be $r_{j}$ and the probability of successfully delivering the package from the depot be $\rho_j$. We define the \textbf{reward-to-risk ratio} of cycle $W_j$ as:
\begin{equation}
    \gamma_j \triangleq \frac{r_{j}\rho_j }{1 - \rho_j^2}
    \label{eq:reward_to_risk_ratio}.
\end{equation}
\label{def:reward_to_risk_ratio}
The numerator captures the expected reward for delivering the package, and the denominator captures the probability of the agent failing during the package delivery cycle.
\end{definition}

\subsection{Ordering of Cycles}
Given a set of package deliveries, we first show that the dispatcher can construct a package delivery plan for each epoch by ordering the packages by their reward-to-risk ratios.

\begin{lemma}
Consider the optimal package delivery plan for epoch $h$, $\mathcal{C}_h^* = \{W^*_{h,1}, W^*_{h,2}, \ldots, W^*_{h,q_h} \}$. Then, the package deliveries are in non-increasing order of their reward-to-risk ratio, i.e. $\forall i < j : \gamma_{h,i} \geq \gamma_{h,j}$.
\label{lem:ordering_of_cycles}
\end{lemma}
\begin{proof}
By examining (\ref{eq:Bellman_eq_a1_mu1}), we can see that the ordering of the deliveries does not change $\bar{\rho}_{h,q_h}$ because the probability of finishing all the cycles in epoch $h$ is a product of the probability of completing each cycle. Therefore, we only need to prove that $E(\mathcal{C}_h | \alpha_h = 1)$ is maximized by ordering the cycles in non-increasing order of their reward-to-risk ratios. We will prove this by contradiction. 
Consider an optimal plan $\mathcal{C}^*_h = \{W_{h,1}, W_{h,2}, ..., W_{h,q_h} \}$ and suppose $\exists i \in \{ 1,..., q_h - 1 \}$ such that $\gamma_{h,i} < \gamma_{h,i+1}$. Using (\ref{eq:reward_to_risk_ratio}), we have 

\begin{align*}
    \frac{r_{h,i} \rho_{h,i}}{1 - \rho^2_{h,i}} < \frac{ r_{h,i+1} \rho_{h,i+1} }{1 - \rho^2_{h,i+1}}. \nonumber
\end{align*}
Which further implies,
\begin{align}
    r_{h,i} \rho_{h,i} +  r_{h,i+1} \rho_{h,i+1} \rho^2_{h,i} < r_{h,i+1} \rho_{h,i+1} +  r_{h,i} \rho_{h,i} \rho^2_{h,i+1} \label{eq:r2r_implication}.
\end{align}

 Now define $\mathcal{C}_{h}$ to be the same as $\mathcal{C}^*_h$ except with the packages in positions $i$ and $i+1$ swapped. Define $\bar{\psi}_{h,i}$ and $\bar{\psi}_{h,i+1}$ to be the probability of receiving rewards $r_{h,i}$ and $r_{h,i+1}$ in the ordering of $\mathcal{C}_h$. Then,
\begin{align}
    E(\mathcal{C}_{h} | & \alpha_h = 1) - E(\mathcal{C}^*_{h} | \alpha_h = 1) = \label{eq:r2r_implication_2} \\ & \sum_{k=0}^{i-1}r_{h,k}\psi_{h,k} + r_{h,i+1}\bar{\psi}_{h,i+1} + r_{h,i}\bar{\psi}_{h,i} \nonumber \\ & + \sum_{k=i+2}^{q_h}r_{h,k}\psi_{h,k} - \theta(1 - \bar{\rho}_{h,q_h}) \nonumber \\ & -\sum_{k=0}^{i-1}r_{h,k}\psi_{h,k}   - r_{h,i}\psi_{h,i} - r_{h,i+1}\psi_{h,i+1} \nonumber \\ & - \sum_{k=i+2}^{q_h}r_{h,k}\psi_{h,k} + \theta(1 - \bar{\rho}_{h,q_h}). \nonumber
\end{align}
 Using the definition of $\psi_{h,i}$ and the ordering of $W_{h,i}$ and $W_{h,i+1}$ in $\mathcal{C}_h$, we know that $\bar{\psi}_{h,i+1} = \rho_{h,i+1}\bar{\rho}_{h,i-1}$ and $\bar{\psi}_{h,i} = \rho_{h,i}\rho_{h,i+1}^2\bar{\rho}_{h,i-1}$. Thus,
\begin{align}
    E(\mathcal{C}_{h} |& \alpha_h = 1) - E(\mathcal{C}^*_{h} | \alpha_h = 1) = \nonumber \\ & r_{h,i+1}{\rho_{h,i+1}}\bar{\rho}_{h,i-1} + r_{h,i}{\rho_{h,i}}{\rho_{h,i+1}^2}\bar{\rho}_{h,i-1} \nonumber \\ - & r_{h,i}{\rho_{h,i}}\bar{\rho}_{h,i-1} - r_{h,i+1}{\rho_{h,i+1}}\rho^2_{h,i}\bar{\rho}_{h,i-1}. \label{eq:r2r_implication_3}
\end{align}

By factoring out $\bar{\rho}_{h,i-1}$ from all the terms and using (\ref{eq:r2r_implication}) in (\ref{eq:r2r_implication_3}), one can see that $E(\mathcal{C}_{h} | \alpha_h = 1) - E(\mathcal{C}^*_{h} | \alpha_h = 1) > 0$ which contradicts the optimality of $\mathcal{C}^*_{h}$.

\end{proof}

Next, we provide a pair of results that specify exactly which packages should be included in the package delivery plan for each epoch.




\begin{lemma}
Consider the optimal package delivery plan $\mathcal{C}^*_h$ that maximizes the Bellman equation (\ref{eq:Bellman_eq_a1_mu1}) for epoch $h$. If the reward-to-risk ratio $\gamma'$ for package $t' \in T$ satisfies the following inequality, then it will be included in the optimal plan $\mathcal{C}^*_h$:
\label{lem:RSPD_good_r2r_lemma}
\begin{equation}
    \gamma' > \theta + V_{h + 1}.
    \label{eq:RSPD_r2r_lemma}
\end{equation}
\end{lemma}
\begin{proof}
We will prove this by contradiction. Suppose $t'$ is not included in the optimal plan $\mathcal{C}^*_h$. Then, the value function  (\ref{eq:Bellman_eq_a1_mu1}) for $\mathcal{C}^*_h$ is
\begin{equation}
    V_h^* =  \sum_{W_{h,i} \in \mathcal{C}^*_h}r_{h,i}\psi_{h,i} + \bar{\rho}_{h,q_h} V_{h + 1} - \theta(1 - \bar{\rho}_{h,q_h}). \nonumber
\end{equation}
Now, let $V'_h$ be the accumulated reward if the delivery cycle $W'$ for $t'$ is included in $\mathcal{C}^*_h$. Let $r'$ and $\rho'$ be the reward and probability of success for delivering $t'$ (and returning), respectively. We relax the ordering of the packages so that $t'$ is delivered after all packaged in $\mathcal{C}^*_h$. By doing this, Lemma \ref{lem:ordering_of_cycles} guarantees this ordering is a lower bound on $V'_h$. We have
\begin{equation}
    V'_h \geq \sum_{W_{h,i} \in \mathcal{C}^*_h}r_{h,i}\psi_{h,i} + r'\psi' +  \rho'\bar{\rho}_{h,q_h} V_{h + 1} - \theta(1 - \rho'\bar{\rho}_{h,q_h}). \nonumber
\end{equation}

Subtracting the value function for $\mathcal{C}'_h$ gives us

\begin{align}
    V'_h - V^*_h & \geq r'{\rho'}\bar{\rho}_{h,q_h} + (V_{h+1}+\theta)(({\rho'})^2\bar{\rho}_{h,q_h} - \bar{\rho}_{h,q_h}) \\
    & =   \bar{\rho}_{h,q_h}(\frac{r'{\rho'}}{(1 - (\rho')^2)} - V_{h + 1} - \theta)(1-(\rho')^2).
    \label{eq:RSPD_r2r_lemma_final_cost}
\end{align}

By applying Definition \ref{def:reward_to_risk_ratio} to (\ref{eq:RSPD_r2r_lemma}) we obtain

\begin{equation}
    \frac{r'{\rho'}}{(1 - (\rho')^2)} > V_{h + 1} + \theta. \label{eq:RSPD_r2r_lemma_expanded}
\end{equation}

But, applying (\ref{eq:RSPD_r2r_lemma_expanded}) to (\ref{eq:RSPD_r2r_lemma_final_cost}) implies that $V'_h - V^*_h > 0$ which contradicts the optimality of $\mathcal{C}^*_h$.

\end{proof}

The derivations in the above lemma also show that including any task whose reward-to-risk ratio is equal to $\theta + V_{h+1}$ will not affect $V_h$. We now show which tasks will be excluded from the optimal plan.

\begin{lemma}
Any package delivery whose reward-to-risk ratio does not satisfy Lemma \ref{lem:RSPD_good_r2r_lemma} is not included in the optimal package delivery plan for epoch $h$, i.e., \label{lem:RSPD_bad_r2r_lemma}
\begin{equation}
    \gamma' < \theta + V_{h + 1}.
    \label{eq:RSPD_bad_r2r_lemma}
\end{equation}
\end{lemma}
\begin{proof}
By Lemma \ref{lem:RSPD_good_r2r_lemma}, $\mathcal{C}^*_h$ should contain all packages in $\mathcal{T}$ whose reward-to-risk ratio does not satisfy (\ref{eq:RSPD_bad_r2r_lemma}). By applying Lemma \ref{lem:ordering_of_cycles}, any package $t'$ whose $\gamma'$ satisfies (\ref{eq:RSPD_bad_r2r_lemma}) will be placed after all package deliveries in $\mathcal{C}^*_h$. Then, by following a similar proof as in Lemma \ref{lem:RSPD_good_r2r_lemma}, it is easy to see that $V_h$ decreases by adding $t'$ to $\mathcal{C}^*_h$.
\end{proof}

\begin{corollary}
By consequence of Lemma \ref{lem:RSPD_good_r2r_lemma} and Lemma \ref{lem:RSPD_bad_r2r_lemma}, the optimal expected reward $V_h$  for epoch $h$ is obtained by the dispatcher assigning all packages whose reward-to-risk ratio satisfies (\ref{eq:RSPD_r2r_lemma}) and ordering them by non-increasing reward-to-risk ratio when calculating $\mathcal{C}^*_h$, i.e.,

\begin{align}
    \mathcal{C}^*_h  = \{ W_{h,1}, W_{h,2}, ..., W_{h,q_h} |  \gamma_{h,1} \geq ... \geq \gamma_{h,q_h} > \theta + V_{h+1} \} \nonumber
\end{align}

\label{cor:RSPD_reward_to_risk_ratio_condition}
\end{corollary}

\begin{corollary}
The package delivery plan for epoch $h$ is always a subset of the package delivery plan for epoch $h+1$, i.e.,

\begin{equation}
    \forall h \in \{ 1,2,...,K-1 \} , \mathcal{C}_h \subseteq \mathcal{C}_{h+1}. \nonumber
\end{equation}

\label{cor:RSPD_subset_package_deliveries}
\end{corollary}

Using Corollary \ref{cor:RSPD_reward_to_risk_ratio_condition} and Corollary \ref{cor:RSPD_subset_package_deliveries}, we construct Algorithm \ref{alg:RSPD} to find the optimal solution to RSPD. Algorithm \ref{alg:RSPD} finishes in $O(n\log n)$ time because sorting the packages by their reward-to-risk ratio requires only $O(n\log n)$. The while loop in line \ref{alg:RSPD:l1} can be accomplished in $O(n)$ time instead of $O(Kn)$ time by using Corollary \ref{cor:RSPD_subset_package_deliveries} to construct package delivery plans from the indices in $\mathcal{T}$ instead of copying packages for each iteration.

\begin{algorithm}[h!]
\SetAlgoLined
\KwIn{An instance of RSPD}
\KwResult{Optimal $\mathcal{C}$ that solves RSPD }
\For{$t_i \in T$}{
 $W_i = \{t_0, t_i, t_0 \}$\;
 Calculate $\gamma_i$ for $W_i$ using equation (\ref{eq:reward_to_risk_ratio}) in Definition \ref{def:reward_to_risk_ratio}\;}
 Let $\mathcal{C}_K$ be the ordered set of all package deliveries $W_i$ such that $\gamma_i > \theta$ and the package deliveries are in non-increasing order of $\gamma_i$\;
 Calculate $V_K$ using (\ref{eq:RSPD_1epoch_expectation}) and (\ref{eq:Bellman_eq_base_a1_mu1}) with $\mathcal{C}_K$\;
 $h = K - 1$\;
 
 \While{$h \ne 0$ \label{alg:RSPD:l1}}{
  Let $\mathcal{C}_h$ be the ordered set of all package deliveries $W_i$ where $\gamma_i > V_{h+1} + \theta$ (non-increasing order of $\gamma_i$) \;
  Calculate $V_h$ using (\ref{eq:RSPD_1epoch_expectation}) and (\ref{eq:Bellman_eq_a1_mu1}) with $\mathcal{C}_h$\;
  $h = h - 1$\;
 }
 \caption{Optimal Solution to RSPD}
 \label{alg:RSPD}
\end{algorithm}
We also note here that the above results (other than Corollary \ref{cor:RSPD_subset_package_deliveries}) would also hold if the set of package delivery plans changed at each epoch. This is shown by the following corollary.
\begin{corollary}
Consider an instance of RSPD where each epoch $h$ has a unique set of packages $\mathcal{T}_h$ available for delivery. Then, Algorithm \ref{alg:RSPD} calculates the optimal package delivery plan in time $Kn + O(n \log n)$.
\end{corollary}
\begin{proof}
Note that Lemmas \ref{lem:ordering_of_cycles}-\ref{lem:RSPD_bad_r2r_lemma} and Corollary \ref{cor:RSPD_reward_to_risk_ratio_condition} still hold for heterogeneous epochs. However, we cannot use Corollary \ref{cor:RSPD_subset_package_deliveries} to remove the $K$ dependence on the time complexity. Analyzing Algorithm \ref{alg:RSPD} with slight modifications to account for heterogeneous epochs and without using Corollary \ref{cor:RSPD_subset_package_deliveries} gives a time complexity of $Kn + O(n \log n)$.
\end{proof}

\subsection{Infinite Horizon Package Delivery}

Up until now we have only considered finite duration missions that are guaranteed to end. However, there may be scenarios where the dispatcher cannot assume the mission will ever end (e.g., logistics companies). We will now define an instance of RSPD with an infinite number of epochs and provide an optimal solution to the problem.

\begin{problem}[Infinite Horizon RSPD (IHRSPD)]
Consider an instance of RSPD with an infinite number of epochs. Then, the problem is to maximize the following expected reward:
 \begin{equation}
    \text{Find   } \text{arg}\max_{\mathcal{C}_1, \mathcal{C}_2,\ldots} \sum_{h = 1}^{\infty} E(\mathcal{C}_{h} | \alpha_h = 1 )P(\alpha_h = 1).
    \label{eq:IHRSPD}
\end{equation}
\label{prob:IHRSPD}
\end{problem}

Despite being an infinite sum, we can show that every package delivery plan for an instance of IHRSPD has a finite expected reward. 

\begin{lemma}
Every instance of IHRSPD has a finite expected reward. \label{lem:IHRSPD_finite_expected_reward}
\end{lemma}
\begin{proof}
Consider any package delivery plan $\mathcal{C}$ whose expected reward is:
\begin{equation}
    E(\mathcal{C}) = \sum_{h = 1}^{\infty} E(\mathcal{C}_{h} | \alpha_h = 1 )P(\alpha_h = 1).
    \label{eq:finite_expected_reward_1}
\end{equation}
First, note that if the plan $\mathcal{C}$ only specifies package deliveries for a finite number of epochs, then the expected reward is trivially finite. Thus, we focus on plans $\mathcal{C}$ that specify deliveries over an infinite number of epochs.  

If for any $\mathcal{C}_h \in \mathcal{C}$, $\mathcal{C}_h = \emptyset$, then $E(\mathcal{C}_{h} | \alpha_h = 1 ) = 0$ and $P(\alpha_{h+1} = 1 | \alpha_{h} = 1) = 1$. This implies that the infinite summation (\ref{eq:finite_expected_reward_1}) does not change if we remove $\mathcal{C}_h$ from $\mathcal{C}$ when $\mathcal{C}_h = \emptyset$. Therefore, we can restrict attention to plans $\mathcal{C}$ which are composed of non-empty package delivery plans for each epoch.

The dispatcher can only choose from a finite number of non-empty package delivery plans for each epoch $h$, i.e., $\mathcal{C}_h \in 2^\mathcal{T}$. By (\ref{eq:RSPD_1epoch_expectation}), the set of all possible values $\mathcal{E}_h$ for the expected reward $E(\mathcal{C}_h|\alpha_h=1)$ during epoch $h$ is also finite and this set does not change with each epoch, i.e., $\mathcal{E}_1 = \mathcal{E}_2 = ...$. By the same argument, the set of possible probabilities of success $\xi_h$ for epoch $h$ also has a finite size and does not change with each epoch. Suppose we knew both the package delivery plan $\mathcal{C}^*_h \in 2^\mathcal{T}$ that maximizes $E(\mathcal{C}^*_h|\alpha_h=1)$ such that $\forall \mathcal{C}_h \in 2^\mathcal{T} , E(\mathcal{C}_h|\alpha_h=1) \leq E(\mathcal{C}^*_h|\alpha_h=1)$ and the package delivery plan $\mathcal{C}'_h$ that maximizes the probability of success $\rho'$ such that $\forall \rho \in \xi_h, \rho \leq \rho'< 1$.

For every $\mathcal{C}_{h} \in \mathcal{C}$, the expected reward $E(\mathcal{C}_{h} | \alpha_h = 1 ) \leq E(\mathcal{C}^*_h|\alpha_h=1)$. Likewise, for every $\mathcal{C}_{h} \in \mathcal{C}$, $P(\alpha_h = 1 | \alpha_{h-1} = 1) \leq \rho'$. By swapping $E(\mathcal{C}_{h} | \alpha_h = 1 )$ for every epoch $E(\mathcal{C}^*_h|\alpha_h=1)$ in (\ref{eq:finite_expected_reward_1}) along with swapping $P(\alpha_h = 1 | \alpha_{h-1} = 1) $ for $ \rho'$ in every epoch when computing $P(\alpha_h = 1)$ in (\ref{eq:finite_expected_reward_1}), we obtain the following upper bound on (\ref{eq:finite_expected_reward_1}) for any package delivery plan $\mathcal{C}$:
\begin{equation*}
    E(\mathcal{C}) \leq \frac{E(\mathcal{C}^*_h|\alpha_h=1)}{1 - \rho'} < \infty. 
\end{equation*}
\end{proof}
Now that we have shown that every instance of IHRSPD has a finite expected reward, we will introduce the Markov Decision Process (MDP) formalism to prove that the optimal package delivery plan is stationary. To show this, we will define a MDP isomorphic to an instance of IHRSPD. Then, using the main result from \cite{10.2307/3690506}, we will show that the optimal expected reward for the MDP and IHRSPD requires choosing the same package delivery plan for each epoch, i.e., a stationary package delivery plan. Now, we provide the following MDP formalism which will likely extend beyond IHRSPD and into many other variations of multi-epoch risk aware task allocation problems. 

\begin{definition} [Markov Decision Process]
 A Markov Decision Process (MDP) is specified by the set $H=(S,A,\{A(x)\},R,P)$ composed of the following structures. The state-space $S$ is a denumerable set, i.e., at most bijective to the natural numbers $\mathbb{N}$. The action set $A$ is a metric space containing all possible actions for every state $x\in S$. For each state $x \in S$, there exists a compact, action set $A(x) \subset A$ of possible actions while in state $x$. $R$ is a reward function defined for each pair of states and actions $\{(x,\mathbf{a}) | x \in S, \mathbf{a} \in A(x)\}$. Finally, the control transition probabilities are specified by a probability function $P = [p_{xy}(\cdot)]$ that defines the probability of transitioning from state $x \in S$ to state $y \in S$, under a given action.
\end{definition}

Every MDP begins in an initial state $X_0$ and progresses to future states in the following manner. Given any current state $X_h = x \in S$, a controller follows some policy $\pi$ that specifies a desired action $A_h = \pi(X_h = x)$ based on the current epoch $h$ and the state $X_h$. Then, a reward $R(X_h,A_h)$ is earned and the state transitions to a new state $X_{h+1} = y \in S$ with probability $p_{xy}(A_h)$. This process repeats indefinitely. Given that this process proceeds indefinitely, if we consider all possible policies $\mathcal{P}$ for a fixed set of reward and probability functions, then there exists an expected reward for every policy $\pi \in \mathcal{P}$ over an infinite horizon.

\begin{definition}[Expected Total Reward for Infinite-Horizon MDP]
Consider a MDP and some policy $\pi$. Then, the expected total reward for the MDP beginning at state $X_0 = x \in S$ under policy $\pi$ is defined as:
\begin{equation}
    V(x,\pi)=E^\pi_x\left[\sum_{h=1}^\infty R(X_h,A_h) \right].
    \label{eq:e_MDP_total}
\end{equation}
\end{definition}

We will construct an isomorphic MDP by following the definition of IHRSPD. Firstly, the agent begins in an alive state $x_s$ and the dispatcher chooses a combination of packages for the agent to deliver while alive. To capture this behavior, construct a MDP whose initial state $x_s$ corresponds to the agent being alive. While the agent is in the alive state, the dispatcher chooses a combination of packages from the set of all packages $\mathcal{T}$. This choice is captured by a policy $\pi$ choosing an action vector $\mathbf{a}$ from the action space $A$, where $A$ is the set of all n-bit strings $A = \{0,1\}^n$ and $n = |\mathcal{T}|$ is the number of tasks in the instance of IHRSPD. For each action vector $\mathbf{a} \in A$, the dispatcher treats each component $a_i$ of $\mathbf{a} \in A$ as a bit that determines if package $t_i \in \mathcal{T}$ is taken for the current epoch $h$. We can define $\mathcal{C}_h$ to be the associated package delivery plan for action vector $\mathbf{a} = A_h$ where the order of the deliveries is given by Lemma \ref{lem:ordering_of_cycles}. Given that the agent attempts to deliver the packages in $\mathcal{C}_h$, there exists $|\mathcal{C}_h| + 1$ chances for the agent to fail in between receiving rewards, i.e., before the first package delivery, between any two package deliveries, or after the last package delivery. Furthermore, given the action vector $\mathbf{a}$, we define the set $\mathcal{F}(\mathbf{a})$ as the set of all n-bit strings that specify the subsets of packages the agent delivered before failure. Then, for each chance of failing, we specify a state $x_{0,\mathbf{f}}$ where $\mathbf{f} \in \mathcal{F}(\mathbf{a})$ specifies the subset of packages the agent was able to deliver before failure. Then, the probability of delivering only the subset of packages specified by $\mathbf{f}$ given the agent attempted to deliver all the packages specified by $\mathbf{a}$ can be captured by a probability function $\phi(\mathbf{f}|\mathbf{a})$.

For all $\mathbf{f} \in \mathcal{F}(\mathbf{a})$ such that $ \mathbf{f} \ne \mathbf{a}$, let $W_j$ be the last package completed by the agent before failing as specified by $\mathbf{f}$. Then, $\phi(\mathbf{f}|\mathbf{a}) = \psi_j(1 - \rho_j\rho_{j+1}) = \psi_j - \psi_{j+1}$ where $\psi_j$ is the probability of completing cycle $W_j$ and $\rho_j$ and $\rho_{j+1}$ are the probabilities of successfully traversing between $t_j$ and the depot, and traversing between the depot and $t_{j+1}$, respectively. If $\mathbf{f} = \mathbf{0}$, then $\phi(\mathbf{f}|\mathbf{a}) = (1 - \rho_{1})$ where $\rho_{1}$ is the probability of delivering the first package specified by $\mathbf{a}$. However, if $\mathbf{f} = \mathbf{a} $, then $\phi(\mathbf{f}|\mathbf{a}) = \psi_{q_h}(1 - \rho_{q_h})$ where $\psi_{q_h}$ is the probability of delivering the last package $t_{q_h}$ and $(1 - \rho_{q_h})$ is the probability of failing between $t_{q_h}$ and the depot. Naturally, if $\mathbf{f} \not \in \mathcal{F}(\mathbf{a})$, then $\phi(\mathbf{f}|\mathbf{a}) = 0$. For the case where the agent delivers all the packages and returns to the depot, we specify the state $x_{1,\mathbf{a}}$ and the probability $\bar{\rho}_{\mathbf{a}} = \bar{\rho}_{q_h}$ where $\bar{\rho}_{q_h}$ is the probability of finishing the last cycle $W_{q_h}$ and returning to the depot. 

\begin{definition}[MIHRSPD Transition Probability Functions]
For all $\mathbf{a} \in A$ and $\mathbf{f} \in \mathcal{F}(\mathbf{a})$,
\begin{equation*}
    p_{x_s,x_{0,\mathbf{f}}}(\mathbf{a}) = \phi(\mathbf{f}|\mathbf{a}) 
\end{equation*}
\begin{equation*}
    p_{x_s,x_{1,\mathbf{a}}} = \bar{\rho}_{\mathbf{a}} 
\end{equation*}
\begin{equation*}
    p_{x_{0,\mathbf{f}},x_d} = p_{x_{1,\mathbf{a}},x_s} = 1.
\end{equation*}
For all other combinations of states $x,y \in S$,
\begin{equation*}
    p_{x,y} = 0.
\end{equation*}
\label{def:MIHRSPD_P_Transition_Functions}
\end{definition}

Using the same conventions, we can construct reward functions given different outcomes in IHRSPD. Namely, if the agent fails while delivering the packages specified by $\mathbf{a}$, then it will receive a partial reward dependent on $x_{0,\mathbf{f}}$. Namely, it will receive the sum of all rewards for the packages it was able to deliver, but then incur a cost $\theta$ for failure. Likewise, if the agent does not fail, i.e.,  $X_h = x_{1,\mathbf{f}}$, then it will receive the sum of all rewards for the packages specified by $\mathbf{f}$ and return to the alive state $x_s$. This is captured by the following reward functions.

\begin{definition}[MIHRSPD Reward Functions]
For all $\mathbf{a} \in A$ and $\mathbf{f} \in \mathcal{F}(\mathbf{a})$,
\begin{equation*}
    R(x_{0,\mathbf{f}}) = \sum_{f_j \in \mathbf{f} : f_j = 1}r_j - \theta
\end{equation*}
\begin{equation*}
    R(x_{1,\mathbf{f}}) = \sum_{f_j \in \mathbf{f} : f_j = 1}r_j.
\end{equation*}
For all other states $x \in S$,
\begin{equation*}
    R(x) = 0.
\end{equation*}
\label{def:MIHRSPD_Reward_Functions}
\end{definition}

Given these probability and reward functions, we can properly define the following MDP which is isomorphic to IHRSPD by definition.

\begin{definition}[Markov-IHRSPD (MIHRSPD)]
An instance of MIHRSPD is given by the set $\mathcal{M} = \{\mathcal{I},H\}$ where $\mathcal{I}$ is an instance of IHRSPD and $H = (S,A,\{A(x)\},R,P)$ is a MDP defined as follows. Define the state-space $S = \{x_s,x_d\} \cup \{ x_i | \forall i \in \{0,1\}\times2^{\mathcal{T}} \}$. Let the action space $A$ be the set of all n-bit strings $A = \{0,1\}^n$, where $n$ is the number of tasks $|\mathcal{T}|$ in $\mathcal{I}$. Define the compact action spaces $\forall x \in S : A(x) = A$. Let the reward functions and probability transition functions be specified by $\mathcal{I}$ using Definition \ref{def:MIHRSPD_Reward_Functions} and Definition \ref{def:MIHRSPD_P_Transition_Functions}. Then, the problem is to find a policy $\pi$ in the set of all policies $\mathcal{P}$ that solves

\begin{equation}
    \text{Find   } \text{arg}\max_{\pi \in \mathcal{P}}E^\pi_x\left[\sum_{h=1}^\infty R(X_h,A_h) \right].
    \label{eq:MIHRSPD_expected_reward}
\end{equation}
\label{def:MIHRSPD}
\end{definition}

Using the main result from \cite{10.2307/3690506}, if the MDP in Definition \ref{def:MIHRSPD} satisfies three simple properties, and every policy has a finite expected reward, i.e., $\forall x \in S$ and $\forall \pi \in \mathcal{P} \implies V(x,\pi) < \infty$, then the optimal policy is stationary, i.e., $\forall h \geq 0, \forall x \in S : X_h = X_{h+1} \implies \pi^*(X_h) = \pi^*(X_{h+1})$.

\begin{theorem} [Existence of Stationary Policies \cite{10.2307/3690506}]
Consider a MDP satisfying the following conditions.  
\begin{enumerate}
    \item For each $x,y \in S$, the mapping $a \mapsto p_{xy}(a)$ is continuous in $a \in A(x)$.
    \item For each $x \in S$, $a \mapsto R(x,a)\in[-\infty,\infty)$ is an upper semicontinuous function of $a \in A(x)$
    \item The number of positive recurrent classes $\mathbf{N}^\tau$ is a continuous function of $\tau \in \mathbb{F}$ where $\mathbb{F}$ is the set of all stationary policies.
\end{enumerate}
Then, there exists an optimal stationary policy $\pi^*$ that maximizes (\ref{eq:e_MDP_total}), i.e., $\forall h \geq 0, \forall x \in S : X_h = X_{h+1} \implies \pi^*(X_h) = \pi^*(X_{h+1})$.
\label{thm:stationary_policy}
\end{theorem}

 A positive recurrent class, as mentioned in Requirement 3 of Theorem \ref{thm:stationary_policy}, is a set of states which will continually be revisited throughout the MDP. Requirement 3 is satisfied if every stationary policy for a given MDP has a continuous number of positive recurrent classes. In the case of MIHRSPD, the number of recurrent classes is constant, i.e., $\mathbf{N}^\tau = 1$. We will prove $\mathbf{N}^\tau = 1$ in the following lemma and apply Theorem \ref{thm:stationary_policy} to show the optimal policy is stationary.

\begin{lemma}
By Theorem \ref{thm:stationary_policy}, the optimal policy for an instance of MIHRSPD is stationary. \label{lem:MIHRSPD_optimal_stationary_policy}
\end{lemma}
\begin{proof}
Consider an instance $\mathcal{M} = \{\mathcal{I},H\}$ of MIHRSPD constructed from $\mathcal{I}$. One can easily verify that the reward functions in Definition \ref{def:MIHRSPD_Reward_Functions}, and the probability transition laws in Definition \ref{def:MIHRSPD_P_Transition_Functions} are continuous functions over $A$, which satisfies Requirements 1 and 2 of Theorem \ref{thm:stationary_policy}. For Requirement 3, if the agent attempts to take any positive number of packages as its stationary policy, i.e., $\mathbf{a} \ne \mathbf{0}$, the agent will eventually fail to deliver all its packages and transition to a state $x_{0,\mathbf{f}}$. After staying in the state $x_{0,\mathbf{f}}$ for one epoch, the state will automatically transition to the failure state, $x_d$, where the MDP will stay forever. Therefore, $x_d$ is considered both an absorbing state and a recurrent state because a given subset of actions will always eventually transition the state to $x_d$ (absorbing) and stay in the state forever (recurrent).  Likewise, if the action vector $\mathbf{a} = \mathbf{0}$ is always chosen as the stationary policy, then $x_s$ can also be considered an absorbing and recurrent state because the probability of transitioning back to $x_s$ is $1$. Then, for any given stationary policy, the number of recurrent classes $\mathbf{N}^\tau = 1$. Finally, using a similar proof as in Lemma \ref{lem:IHRSPD_finite_expected_reward}, one can show that the expected reward for $\mathcal{M}$ is finite because $\mathcal{M}$ is isomorphic to $\mathcal{I}$. Therefore, MIHRSPD will always satisfy Theorem \ref{thm:stationary_policy}, which implies the optimal policy is stationary. 
\end{proof}

\begin{corollary}
The optimal package delivery plan $\mathcal{C}$ for an instance of IHRSPD is stationary, i.e., $\forall \mathcal{C}_h \in \mathcal{C}: \mathcal{C}_h = \mathcal{C}_{h+1}$.
\end{corollary}
\begin{proof}
By Lemma \ref{lem:MIHRSPD_optimal_stationary_policy}, the optimal policy for a MDP $\mathcal{M}$ constructed from instance of IHRSPD $\mathcal{I}$ is stationary. Because $\mathcal{M}$ is by definition isomorphic to $\mathcal{I}$, the optimal package delivery plan for $\mathcal{I}$ is stationary.
\end{proof}

Given we know the optimal package delivery plan is stationary, we will now find the optimal stationary package delivery plan. To do this, we extend the notion of reward-to-risk ratio to apply to package delivery plans using epoch risk ratios.

\begin{definition}[Epoch Risk Ratio]
Given a package delivery plan $\mathcal{C}_h$ for epoch $h$ during an instance of RSPD or IHRSPD, let $\bar{\rho}_{q_h}$ be the probability of finishing the last cycle $W_{q_h}$ in $\mathcal{C}_h$. Then, we denote the epoch risk ratio for $\mathcal{C}_h$ as $\epsilon_h$ and compute it by:
\begin{equation*}
   \epsilon_h = \frac{E(\mathcal{C}_{h} | \alpha_h = 1 )}{1 - \bar{\rho}_{q_h}}
\end{equation*}
\end{definition}

Given the package delivery plan is stationary, the following lemma characterizes the expected reward using the epoch risk ratio.

\begin{lemma}
Consider an instance of IHRSPD with a stationary package delivery plan, i.e., $\mathcal{C}_{1} = \mathcal{C}_2 = \ldots $. Then, the expected reward for IHRSPD is given by the epoch risk ratio for the package delivery plan for any epoch $\mathcal{C}_{h}$:
\begin{equation*}
    \frac{E(\mathcal{C}_{h} | \alpha_1 = 1 )}{1 - \bar{\rho}_{q_h}}. 
\end{equation*}
\label{lem:IHRSPD_stationary_plan}
\end{lemma}
\begin{proof}
Because the package delivery plans for each epoch are equal, $\mathcal{C}_{1} = \mathcal{C}_2 = \ldots$, the expected reward for each epoch (\ref{eq:RSPD_1epoch_expectation}) will also be equal, i.e., $E(\mathcal{C}_{1} | \alpha_1 = 1) = E(\mathcal{C}_{2} | \alpha_2 = 1) = \ldots $. Furthermore, the last cycles in each package delivery plan will also be equal $q_1 = q_2 = ... = p$ for some $p$. By applying the same assumption to (\ref{eq:RSPD_q_h}), we can see that $P(\alpha_h = 1) = (\bar{\rho}_{q_h})^{h-1} = (\bar{\rho}_{p})^{h-1}$. Given (\ref{eq:IHRSPD}), we can factor out the expectation $E(\mathcal{C}_{1} | \alpha_1 = 1)$ from every epoch and rewrite (\ref{eq:IHRSPD}) to be a geometric series which yields the epoch risk ratio. More specifically, we have

\begin{align}
    \sum_{h = 1}^{\infty} E(\mathcal{C}_{h} | \alpha_h = 1 )P(\alpha_h = 1) &= E(\mathcal{C}_{1} | \alpha_1 = 1 ) \sum_{h = 1}^{\infty} (\bar{\rho}_p)^{(h-1)} \nonumber \\ & = \frac{E(\mathcal{C}_{1} | \alpha_1 = 1 )}{1 - \bar{\rho}_{p}}. \nonumber
\end{align}
\end{proof}

\begin{lemma}
The epoch risk ratio $\epsilon_h$ for a package delivery plan $\mathcal{C}_h$ during epoch $h$ is maximal when $\mathcal{C}_h = \{W_{\text{max}}\}$ where $W_{\text{max}}$ is the package delivery with the maximal epoch risk ratio, $\epsilon_{\text{max}}$. \label{lem:optimal_package_delivery_horizon}
\end{lemma}
\begin{proof}
We will prove this by contradiction. Consider the optimal package delivery plan $\mathcal{C}^*_h$ whose epoch risk ratio is greater than $\epsilon_{\text{max}}$. However, for each individual package delivery in $W_i \in \mathcal{C}^*_h$, its reward-to-risk ratio is less than  $\gamma_\text{max}$ (the largest reward-to-risk ratio for any individual task). This implies
\begin{align}
    \frac{E(\mathcal{C}^*_h | \alpha_h = 1)}{1 - \bar{\rho}_{q_h}} &> \frac{E(\mathcal{C}_h | \alpha_h = 1)}{1 - (\rho_1)^2} \nonumber \\   
    \Rightarrow \frac{\sum_i{r_i\psi_i}}{1 - \bar{\rho}_{q_h}} &> \frac{r_1 \rho_1}{1 - (\rho_1)^2} \nonumber \\
    \Rightarrow  \sum_i{r_i\psi_i} + r_1\rho_1\bar{\rho}_{q_h} &> r_1\rho_1 + \sum_i{r_i\psi_i} (\rho_1)^2.
    \label{lem:eqfinal:optimal_package_delivery_horizon}
\end{align}

Inequality (\ref{lem:eqfinal:optimal_package_delivery_horizon}) claims that the expected reward for a package delivery plan where $W_{\text{max}}$ is delivered after all packages in $\mathcal{C}^*_h$ is optimal compared to delivering $W_{\text{max}}$ before all packages in $\mathcal{C}^*_h$. However, if $W_{\text{max}}$ has a higher reward-to-risk ratio than any package in $\mathcal{C}^*_h$, then Lemma \ref{lem:ordering_of_cycles} proves this is suboptimal. 
\end{proof}

\begin{corollary}
By an extension of Lemma \ref{lem:optimal_package_delivery_horizon}, the optimal stationary package delivery plan $\mathcal{C}$ that solves an instance $\mathcal{I}$ of IHRSPD is given by $\mathcal{C} = \{\forall h > 0 : \mathcal{C}_h = \{W_{\text{max}}\}\}$ where $W_{\text{max}}$ is the package delivery with highest reward-to-risk ratio in $\mathcal{I}$.
\label{cor:optimal_IHRSPD}
\end{corollary}

The above result shows that in the case of the infinite horizon package delivery problem, the optimal strategy is to deliver only the most valuable package in each epoch.  Thus, one can easily construct an $O(n)$ time algorithm to solve IHRSPD by searching for the package with greatest reward-to-risk ratio.

\addtolength{\textheight}{-3cm}   


\section{CONCLUSIONS AND FUTURE WORK}

In this work, we formulate the Risk-Aware Single Agent Delivery (RSPD) problem where a dispatcher assigns package deliveries to a single agent based on both the reward associated with the package and the probability of successfully surviving the delivery. If the agent fails mid-delivery, a cost is incurred and the agent cannot make future deliveries. Therefore, the dispatcher must weigh the reward from a delivery against the loss of all future rewards across multiple epochs of the problem. We solve variants of this problem, one where there is a finite number of epochs (finite horizon) and one where there is an infinite number of epochs (infinite horizon). For the finite horizon case, we provide an optimal algorithm with time-complexity $O(n \log n)$. For the infinite horizon problem, we construct an isomorphic Markov Decision Process (MDP) to prove the optimal package delivery plan is to deliver the same package forever. Finding the optimal solution to the infinite horizon case only takes $O(n)$ time.  Leveraging this connection to address multi-agent variants of our problem may be of interest for future work.

\section{ACKNOWLEDGMENTS}

We would like to acknowledge Lintao Ye for helpful discussions on greedy algorithms.


\bibliographystyle{abbrv}
\bibliography{references}

\end{document}